\documentclass[10pt]{article}
\usepackage{fullpage, amsmath, amssymb, amsfonts, amsthm, stmaryrd, url, hyperref, indentfirst, bm, bbm, extarrows, tikz-cd}
\usepackage[all]{xy}

\newtheorem{theorem}{Theorem}[section]
\newtheorem{lemma}[theorem]{Lemma}

\newtheorem{conj}[theorem]{Conjecture}
\newtheorem{prop}[theorem]{Proposition}
\theoremstyle{definition}
\newtheorem{defn}[theorem]{Definition}

\newtheorem{rmk}[theorem]{Remark}

\numberwithin{equation}{theorem}

\newcommand{\AAA}{\mathbb{A}}
\newcommand{\CC}{\mathbb{C}}

\newcommand{\FF}{\mathbb{F}}

\newcommand{\QQ}{\mathbb{Q}}

\newcommand{\ZZ}{\mathbb{Z}}

\newcommand{\calD}{\mathcal{D}}

\newcommand{\calG}{\mathcal{G}}

\newcommand{\calL}{\mathcal{L}}
\newcommand{\calM}{\mathcal{M}}
\newcommand{\calO}{\mathcal{O}}

\newcommand{\calY}{\mathcal{Y}}

\newcommand{\gothc}{\mathfrak{c}}

\newcommand{\gothl}{\mathfrak{l}}
\newcommand{\gothm}{\mathfrak{m}}

\newcommand{\gothp}{\mathfrak{p}}

\newcommand{\gothq}{\mathfrak{q}}

\DeclareMathOperator{\Art}{Art}

\DeclareMathOperator{\diag}{diag}
\DeclareMathOperator{\End}{End}

\DeclareMathOperator{\Frob}{Frob}
\DeclareMathOperator{\Fss}{F-ss}

\DeclareMathOperator{\GL}{GL}

\DeclareMathOperator{\Hom}{Hom}
\DeclareMathOperator{\image}{image}

\DeclareMathOperator{\Irr}{Irr}

\DeclareMathOperator{\rec}{rec}

\DeclareMathOperator{\Res}{Res}

\DeclareMathOperator{\SL}{SL}

\DeclareMathOperator{\Ss}{ss}

\DeclareMathOperator{\WD}{WD}
\DeclareMathOperator{\WDrep}{WDrep}

\title{An Ordinary Rank-Two Case of Local-Global Compatibility for Automorphic Representations of Arbitrary Weight Over CM Fields}
\author{Yuji Yang}
\date{}

\begin{document}
\maketitle
\begin{abstract}
We prove a rank-two potential automorphy theorem for mod $l$ representations satisfying an ordinary condition. Combined with an ordinary automorphy lifting theorem from \cite{ACC+18}, we prove a rank-two, $p\ne l$ case of local-global compatibility for regular algebraic cuspidal automorphic representations of arbitrary weight over CM fields that is $\iota$-ordinary for some $\iota:\overline{\QQ}_l\xrightarrow{\sim}\CC$.
\end{abstract}
\section{Introduction}

The goal of this work is to prove a new case of local-global compatibility for automorphic Galois representations, which we will introduce in a general sense below.

Let $F$ be a number field. Let $\pi=\otimes'_v\pi_v$ be an algebraic cuspidal automorphic representation of $\GL_n(\AAA_F)$, where $\pi_v$ is an irreducible representation of $\GL_n(F_v)$ for each place $v$ in $F$. Let $l$ be a rational prime and fix an isomorphism $\iota:\overline{\QQ}_l\xrightarrow{\sim}\CC$. For each $v$ the local Langlands correspondence gives a bijection
$$\rec_{F_v}:\Irr(\GL_n(F_v))\to\WDrep_n(W_{F_v}),$$
normalized as in \cite{HT01}, where $\Irr(\GL_n(F_v))$ is the set of isomorphism classes of irreducible smooth admissible representations of $\GL_n(F_v)$ over $\CC$, and $\WDrep_n(W_{F_v})$ is the set of isomorphism classes of $n$-dimensional Frobenius semisimple Weil--Deligne representations of $W_{F_v}$ over $\CC$. Recall that a Weil--Deligne representation is a pair $(r,N)$, where $r$ is a representation of the Weil group $W_{F_v}$ of $F_v$ and $N$ is a nilpotent endomorphism of the representation space that intertwines with $r$. We can apply the local Langlands correspondence map to a twist of $\pi_v$ to get $$\rec_{F_v}(\pi_v\otimes|\det|^{\frac{1-n}{2}}),$$ 
a Weil--Deligne representation of $W_{F_v}$ over $\CC$.

Conjecturally there exists a continuous semisimple Galois representation $r_\iota(\pi):G_F\to \GL_n(\overline{\QQ}_l)$ attached to $\pi$ and $\iota$. This restricts to $G_{F_v}$, and if $v\nmid l$, by Grothendieck's $l$-adic monodromy theorem the restriction determines a Weil--Deligne representation $\WD(r_\iota(\pi)|_{G_{F_v}})$. We can take its Frobenius semisimplification (recall that the Frobenius semisimplification $(r,N)^{\Fss}=(r^{\Ss},N)$) and then base change to $\CC$ to get $$\WD(r_\iota(\pi)|_{G_{F_v}})^{\Fss}\otimes_\iota\CC,$$
another Weil--Deligne representation of $W_{F_v}$ over $\CC$.

The local-global compatibility conjecture says that these two objects are isomorphic:
\begin{conj}[Local-global compatibility]
For all finite place $v$ of $F$, we have
$$\WD(r_\iota(\pi)|_{G_{F_v}})^{\Fss}\otimes_\iota\CC\cong\rec_{F_v}(\pi_v\otimes|\det|^{\frac{1-n}{2}}).$$
\end{conj}

\begin{rmk}
We only defined $\WD$ when $v\nmid l$. We call it the $p\ne l$ case, assuming that $p$ is the residue characteristic of $F_v$. We can also define a Weil--Deligne representation in the $p=l$ case.
\end{rmk}

From now on we will focus on the case where $F$ is a CM field and $p\ne l$. Fix a rational prime $l$ and an isomorphism $\iota: \overline{\QQ}_l\xrightarrow{\sim}\CC$ as above. Let $\pi$ be a regular algebraic cuspidal automorphic representation of $\GL_n(\AAA_F)$. If $\pi$ is polarizable, local-global compatibility has been proved by Caraiani \cite{Car12} (and by Taylor–Yoshida in less generality \cite{TY07}). For non-polarizable $\pi$, we mention the following two remarkable results. In \cite{HLTT16}, Harris--Lan--Taylor--Thorne showed the existence of $r_\iota(\pi)$. Moreover, if the other prime $p$ satisfies some unramified conditions, then local-global compatibility holds at any finite place $v\mid p$ in $F$. In \cite{Var14}, Varma showed that for any finite place $v\nmid l$, local-global compatibility holds up to semisimplification (recall that the semisimplification $(r,N)^{\Ss}=(r^{\Ss},0)$). The question is then how to deal with the monodromy.

The work of Allen--Newton \cite[Theorem~4.1]{AN19} addresses the question when $\pi$ is of weight $0$ in the $\GL_2$ case. They prove that there is set of primes $l$ of Dirichlet density one such that local-global compatibility holds for all finite places $v\nmid l$ in $F$. The strategy of proof is to combine an automorphy lifting theorem with Taylor's potential automorphy method. The idea to approach the local-global compatibility conjecture using the automorphy lifting theorems is inspired by \cite{Luu15}. The proof of \cite[Theorem~4.1]{AN19} uses the Fontaine--Laffaille automorphy lifting theorem \cite[Theorem~6.1.1]{ACC+18}, which does not allow a change of weights, and Taylor's potential automorphy method gives an input of weight $0$, so they together require $\pi$ to be of weight $0$. However, the authors mentioned that the weight $0$ condition might be removed if we use the ordinary automorphy lifting theorem \cite[Theorem~6.1.2]{ACC+18} instead, at the cost of imposing some ordinary assumptions on $\pi$ at $l$. Here we investigate this question. We will prove the following result.

\begin{theorem}
Let $F$ be a CM field and let $\pi$ be a regular algebraic cuspidal automorphic representation of $\GL_2(\AAA_F)$. Suppose that 
\begin{enumerate}
    \item $l$ is an odd prime,
    \item $\pi$ is $\iota$-ordinary for some $\iota:\overline{\QQ}_l\cong\CC$,
    \item $\overline{r}_{\iota}(\pi)$ is decomposed generic, $\overline{r}_{\iota}(\pi)(G_{F(\zeta_l)})$ is enormous, and there is a $\sigma\in G_F-G_{F(\zeta_l)}$ such that $\overline{r}_{\iota}(\pi)(\sigma)$ is a scalar.
\end{enumerate}
Then for any finite place $v\nmid l$ in $F$, we have
$$\WD(r_{\iota}(\pi)|_{G_{F_v}})^{\Fss}\otimes_\iota\CC\cong\rec_{F_v}(\pi_v\otimes|\det|^{-1/2}).$$
\end{theorem}

The idea of proof is similar to that of \cite[Theorem~4.1]{AN19}. By the main result of \cite{Var14}, it suffices to show that when $\pi_v$ is special (i.e. a twist of Steinberg), $r_{\iota}(\pi)|_{G_{F_v}}$ has nontrivial monodromy. Assume the monodromy is trivial. After a base change we may assume that $\pi_v$ is an unramified twist of Steinberg. We then want to find an automorphic representation $\pi_1$ that is unramified at $v$, such that $\overline{r}_{\iota}(\pi_1)\cong\overline{r}_{\iota}(\pi)$ after some restriction. This step is done by a potential automorphy theorem \cite[Theorem~3.9]{AN19}. Then the automorphy lifting theorem will give an automorphic representation $\Pi$ with $r_{\iota}(\Pi)\cong r_{\iota}(\pi)$ after some restriction, from which we can deduce a contradiction with the unramified local-global compatibility known from \cite{HLTT16,Var14}.

However, since we use the ordinary automorphy lifting theorem instead of the Fontaine--Laffaille one, there are some major differences. First, the ordinary automorphy lifting theorem does not specify the weights of the input automorphic representations $\pi$ and $\pi_1$, so we are able to handle regular algebraic cuspidal automorphic representations of arbitrary weight. Also, we need to prove an ordinary version of potential automorphy theorem (Theorem \ref{g3.9} proved below) to construct the automorphic representation $\pi_1$ which in some sense inherits ordinariness. More precisely, we impose an ordinary condition (Definition \ref{GoodOrdinary} below) on the Galois representation and we show that the automorphic representation we construct is $\iota$-ordinary. Finally, since we are not in the Fontaine--Laffaille case, we allow the prime $l$ to ramify in $F$.

\hspace*{\fill} \\
\textbf{Notations.} Let $F$ be a CM field. Let $l$ be an odd prime and let $v$ be a prime of $F$ over $l$. Let $k/\FF_l$ be a finite field. Let $\calO$ be a discrete valuation ring finite over $W(k)$ with residue field $k$. Let $\calD_{\calO/\ZZ_l}$ denote the different. Let $\gothl$ be a uniformiser in $\calO$. Let $\epsilon_l$ be the $l$-adic cyclotomic character. We normalize our Hodge--Tate weights such that $\epsilon_l$ has all labelled Hodge--Tate weights equal to $-1$.

Let $M$ be a totally real field. We recall the definition of $M$-Hilbert--Blumenthal abelian varieties ($M$-HBAV in short). Let $S$ be a scheme and let $A$ be an abelian scheme over $S$ equipped with a given embedding of rings (real multiplication) $\iota:\calO_M\hookrightarrow\End(A/S)$. Let $\calM_A$ be the polarization module of $A$, i.e., the module of  $\calO_M$-linear symmetric homomorphisms $A\to A^{\vee}$, and let $\calM_A^+$ be the positive cone of polarizations. An $M$-HBAV is an abelian scheme with real multiplication $(A,\iota)$ such that the natural map $A\otimes_{\calO_M}\calM_A\to A^{\vee}$ is an isomorphism. Let $\gothc$ be a nonzero fractional ideal of $M$ and let $\gothc^+$ be the totally positive elements in $\gothc$. A $\gothc$-polarization of $A$ is an isomorphism $j:\gothc\xrightarrow{\sim}\calM_A$ of $\calO_M$-modules such that $j(\gothc^+)=\calM_A^+$. See \cite{DP94} for more details.

By a divisible $\calO$-module over a scheme $S$ we mean an $l$-divisible group $\calG/S$ with a ring homomorphism $\calO\to\End_S\calG$. Let $K$ be a local field with residue characteristic $l$. We say a divisible $\calO$-module $\calG/\calO_K$ is ordinary if there is an embedding $\mu_{l^\infty}\otimes_{\ZZ}\calO\hookrightarrow\calG_{\calO_K^{\text{nr}}}$.


Finally we recall the definition of a Galois representation being ordinary of some weight. Let $\ZZ^n_+$ be the set of $n$-tuples with coordinates in decreasing order. Take a continuous representation $\rho:G_F\to\GL_n(\overline{\QQ}_l)$ and let $\lambda\in(\ZZ^n_+)^{\Hom(F,\overline{\QQ}_l)}$. We say that $\rho$ is ordinary of weight $\lambda$ if for each $v\mid l$ in $F$
$$\rho|_{G_{F_v}}\cong\left(\begin{array}{cccc}
    \psi_1 & * & * & * \\
    0 & \psi_2 & * & * \\
    \vdots & \ddots & \ddots & * \\
    0 & \cdots & 0 & \psi_n
\end{array}\right)$$
where for each $1\le i\le n$, $\psi_i:G_{F_v}\to \overline{\QQ}_l^\times$ is a continuous character such that $$\psi_i(\Art_{F_v}(\sigma))=\prod_{\tau:F_v\hookrightarrow\overline{\QQ}_l}\tau(\sigma)^{-(\lambda_{\tau,n-i+1}+i-1)}$$ for all $\sigma$ in some open subgroup of $\calO_{F_v}^\times$.

\section{An Important Lemma}
In this section we prove a lemma constructing certain $\gothl$-divisible groups that is important for our potential automorphy result (Theorem \ref{g3.9}). We first make the following definition to simplify our statement. 

\begin{defn}\label{GoodOrdinary}
Let $\overline{r}:G_{F_v}\to\GL_2(k)$ be a continuous representation. We say $\overline{r}$ is \textbf{good ordinary} if $$\overline{r}\cong\left(\begin{array}{cc}
    \overline{\psi}\overline{\epsilon}_l & \overline{c} \\
    0 & \overline{\psi}^{-1}
\end{array}\right)$$ for some unramified character $\overline{\psi}$ and extension class $\overline{c}$, such that either $\overline{\psi}^2\ne1$ or $\overline{\psi}^2=1$ and $\overline{c}$ is peu ramifi\'e (\cite[\S2.4]{Ser87}).
\end{defn}

Now we state our important lemma, which can be viewed as an ordinary version of \cite[Lemma~3.8]{AN19}. Since we allow $l$ to ramify in $F$, we use Breuil--Kisin theory instead of Fontaine--Laffaille theory.

\begin{lemma}\label{g3.8}
Let $\overline{r}:G_{F_v}\to\GL_2(k)$ be a continuous representation such that:
\begin{itemize}
    \item $\det\overline{r}=\overline{\epsilon}_l$,
    \item there is a crystalline lift $r:G_{F_v}\to\GL_2(\calO)$ with labelled Hodge--Tate weights all equal to \{$-1,0$\}.
\end{itemize}
Let $V_{\overline{r}}$ be the underlying \'etale $k$-vector space scheme over $F_v$ of $\overline{r}$. Let $\gothl$ be a uniformiser in $\calO$. Then there exists a divisible $\calO$-module $\calG$ over $\calO_{F_v}$ of height $2[\calO:\ZZ_l]$, equipped with an $\calO$-linear symmetric isomorphism $\lambda:\calG\cong\calG^\vee\otimes\calD_{\calO/\ZZ_l}$, and an isomorphism $i:V_{\overline{r}}\cong\calG[\gothl]_{F_v}$ such that $i^\vee\circ\lambda[\gothl]_{F_v}\circ i$ is the isomorphism induced by the standard symplectic pairing on $V_{\overline{r}}$. 

Moreover, if $\overline{r}$ is good ordinary, then $\calG$ can be chosen to be ordinary.
\end{lemma}

\begin{proof}

\textbf{Step 1: Reduction.} By assumption we have $\overline{r}:G_{F_v}\to\GL_2(k)$ such that  $\det\overline{r}=\overline{\epsilon}_l$, with a crystalline lift $r:G_{F_v}\to\GL_2(\calO)$ with all labelled Hodge--Tate weights in $\{-1,0\}$, so $\det r|_{I_{F_v}}=\epsilon_l$. Then there is an unramified character $\psi:G_{F_v}\to\calO^\times$ given by $\psi^2=(\det r)^{-1}\epsilon_l$, such that $r':=r\otimes\psi$ has $\det r'=\epsilon_l$ and it is another crystalline lift of $\overline{r}$ satisfying the assumptions. Thus we may assume without loss of generality that $\det r=\epsilon_l$.

\hspace*{\fill} \\
\textbf{Step 2: Construction.} The existence of the $l$-divisible group $\calG$ over $\calO_{F_v}$ follows from \cite[Corollary~6.2.3]{SW13}. For a constructive proof we use Breuil--Kisin theory and refer to \cite[Lemma~2.1.15 and Corollary~2.2.6]{Kis06}. Also see \cite[Lemma~11.2.10 and Theorem~12.3.2]{BC09}. From the construction \cite[Corollary~2.2.6]{Kis06} or \cite[Theorem~12.3.2]{BC09} we know that there is a natural isomorphism of lattices $r\cong T_l(\calG)$.

\hspace*{\fill} \\
\textbf{Step 3: Properties.} We finally check that $\calG$ meets all the conditions we demand. First of all, there is an obvious $\calO$-action on the representation space $\calO^2$ and by functoriality this gives the $\calO$-action on $\calG$; in other words, we get $$\calO\hookrightarrow\End\calG.$$ The height of $\calG$ is equal to the $\ZZ_l$-rank of $T_l(\calG)$, which is $2[\calO:\ZZ_l]$. 
Next, we compute that $$V_{\overline{r}}=V_{r\text{ mod }\gothl}\cong T_l(\calG)\text{ mod }\gothl=\calG[\gothl]_{F_v}(\overline{F}_v).$$
It remains to produce $\lambda$. Take $T=\calO^2$ with $G_{F_v}$-action given by $r$. We have the (Galois equivariant) standard symplectic pairing $\langle\cdot,\cdot\rangle:\calO^2\times\calO^2\to\calO(1)$, so $T\cong\Hom_\calO(T,\calO(1))$. The trace pairing from $\calO$ to $\ZZ_l$ gives $\Hom_\calO(T,\calO(1))\otimes\calD^{-1}_{\calO/\ZZ_l}\cong\Hom_{\ZZ_l}(T,\ZZ_l(1))$, and hence there is an isomorphism of Tate modules $$T_l(\calG)\cong T_l(\calG^\vee)\otimes\calD_{\calO/\ZZ_l}\cong T_l(\calG^\vee\otimes\calD_{\calO/\ZZ_l}).$$ 
Hence by the main result of \cite{Tat67} we get an isomorphism of $l$-divisible groups $\calG\cong\calG^\vee\otimes\calD_{\calO/\ZZ_l}$. The compatibility with the standard symplectic pairing on $V_{\overline{r}}$ follows by construction.

Now we let $\overline{r}$ be good ordinary. We first point out that it is a stronger assumption: obviously $\det(\overline{r})=\overline{\epsilon}_l$, and \cite[Lemma~6.1.6]{BLGG12} implies that there is a potentially crystalline lift of $\overline{r}$ with labelled Hodge--Tate weights all equal to $\{-1,0\}$. Note that we normalize the representation differently, but the lemma still holds since the proof only cares about the quotient of the diagonal characters. Moreover, by Definition \ref{GoodOrdinary}, we are in Case $1$ or Case $2$b in the proof of \cite[Lemma~6.1.6]{BLGG12}, where we have some control on the lifts of the diagonal characters. Since we further demand the character $\overline{\psi}$ to be unramified in Definition \ref{GoodOrdinary}, we can guarantee that the lifts of the characters are unramified instead of potentially unramified, and hence deduce that there is a crystalline lift of $\overline{r}$ (rather than potentially crystalline). 

Therefore it remains to show that $\calG$ can be chosen to be ordinary. By \cite[Lemma~6.1.6]{BLGG12} we know that the lift $r$ is of the form $$r\cong\left(\begin{array}{cc}
    \psi\epsilon_l & c \\
    0 & \psi^{-1}
\end{array}\right),$$ which is ordinary, and hence we know $r|_{I_{F_v}}$ has a subrepresentation $\epsilon_l\otimes_\ZZ\calO$. Note that $r\cong T_l(\calG)$ and $T_l(\mu_{l^\infty})\cong\epsilon_l$. By the main result of \cite{Tat67} we get an embedding $\mu_{l^\infty}\otimes_\ZZ\calO\hookrightarrow\calG_{\calO_{F_v}^\text{nr}}$, which means that $\calG$ is ordinary.
\end{proof}

\section{Potential Automorphy}
We first state an ordinary version of \cite[Proposition~3.6]{AN19}. 

\begin{prop}\label{g3.6}
Let $k$ be an algebraically closed field of characteristic $l$. Let $M$ be a totally real field and let $\gothl$ be a prime in $M$ over $l$. Suppose that $(\calG,\lambda)$ is a divisible $\calO_{M,\gothl}$-module over $k$ of height $2[M_\gothl:\QQ_l]$ equipped with an $\calO_{M,\gothl}$-linear symmetric isomorphism $\lambda:\calG\cong\calG^\vee\otimes\calD_{\calO_{M,\gothl}/\ZZ_l}$.

Then there exists an $M$-HBAV $(A,\iota,j)$ over $k$ with $\calD_M^{-1}$-polarization and an isomorphism $i:A[\gothl^\infty]\cong\calG$ compatible with $\calO_{M,\gothl}$-actions on both sides such that $i^\vee\circ\lambda\circ i=j(1)$.

Moreover, if $\calG$ is ordinary, then $A$ can be chosen to be ordinary.
\end{prop}

\begin{proof}
We hope to adjust the polarization so that we can apply \cite[Proposition~3.6]{AN19}. Since the isomorphism class of the polarization only depends on its image in the strict class group, we may take a totally positive element $r\in\calO_M^+$ such that for any $\gothl'\mid l$, $r$ has the same $\gothl'$-adic valuation as $\calD_M$. Then for each $\gothl'\mid l$, multiplication by $r^{-1}$ induces an isomorphism $\calD_{\calO_{M,\gothl'}/\ZZ_l}\to\calO_{M,\gothl'}$. In particular, $\lambda'=\lambda\otimes r^{-1}$ would be an isomorphism $\calG\to\calG^\vee$.

Let $\gothc=r\calD_M^{-1}$. By our choice of $r$, $\gothc$ is a fractional ideal of $M$ that contains $\calO_M$ and is coprime to $l$. Applying \cite[Proposition~3.6]{AN19}, we get an $M$-HBAV with $\gothc$-polarization $(A,\iota,j')$ and an $\calO_{M,\gothl}$-equivariant isomorphism $i:A[\gothl^\infty]\cong\calG$ such that $i^\vee\circ\lambda'\circ i=j'(1)$. Here $j':\gothc\to\calM_A$ gives the $\gothc$-polarization, and hence $j=j'\circ r:\calD_M^{-1}\to\calM_A$ is the desired $\calD_M^{-1}$-polarization. We can check that $j(1)=i^\vee\circ\lambda'\circ i\circ r=i^\vee\circ(\lambda'\circ r)\circ i=i^\vee\circ\lambda\circ i$ since $i$ is $\calO_{M,\gothl}$-equivariant.

It remains to show the ordinary statement. In the proof of \cite[Proposition~3.6]{AN19} they obtain a $\calD_M^{-1}$-polarized $M$-HBAV $A_0$ such that $A_0[\gothl^\infty]\cong\calG$, and in order to apply Dieudonn\'e theory to $A_0[l^\infty]$, they set $\calG_l=\calG\times\prod_{\gothl'\ne\gothl, \gothl'\mid l}A_0[\gothl'^\infty]$. Now since $\calG$ is ordinary, we have that $A_0[\gothl^\infty]$ is ordinary, and we demand $A_0[\gothl'^\infty]$ to be ordinary for all $\gothl'\mid l$, $\gothl'\ne\gothl$. This will yield an ordinary $M$-HBAV $A$ with desired properties.
\end{proof}

We then state our potential automorphy theorem, which is an ordinary version of \cite[Theorem~3.9]{AN19}.

\begin{theorem}\label{g3.9}
Suppose that $F$ is a CM field, $l$ is an odd prime and $k/\FF_l$ finite. Let $\calO$ be a discrete valuation ring finite over $W(k)$ with residue field $k$. Let $\overline{\rho}:G_F\to\GL_2(k)$ be a continuous absolutely irreducible representation such that 
\begin{itemize}
    \item $\det(\overline{\rho})=\overline{\epsilon}_l^{-1}$;
    \item for each $v\mid l$, $\overline{\rho}|_{G_{F_v}}$ admits a crystalline lift $\rho_v:G_{F_v}\to\GL_2(\calO)$ with all labelled Hodge–Tate weights equal to $\{0, 1\}$.
\end{itemize}

Suppose moreover that $F^\text{avoid}/F$ is a finite extension. Then we can find
\begin{itemize}
    \item a finite CM extension $F_1/F$ that is linearly disjoint from $F^\text{avoid}$ over $F$, such that if $v\mid l$ in $F$, then $v$ is unramified in $F_1$;
    \item a regular algebraic cuspidal automorphic representation $\pi$ for $GL_2(\AAA_{F_1})$ of weight $0$, unramified at places above $l$;
    \item an isomorphism $\iota:\overline{\QQ}_l\xrightarrow{\sim}\CC$
    such that (composing $\overline{\rho}$ with some embedding $k\hookrightarrow\overline{\FF}_l$)
    $$\overline{r}_\iota(\pi)\cong\overline{\rho}|_{G_{F_1}}.$$
\end{itemize}

Moreover, we can ensure the following:
\begin{enumerate}
    \item If $\overline{v}_0\nmid l$ is a finite place of $F^+$, then we can moreover find $F_1$ and $\pi$ as above with $\pi$ unramified above $\overline{v}_0$.
    \item If for each $v\mid l$, $\overline{\rho}^\vee|_{G_{F_v}}$ is good ordinary, then $\pi$ is furthermore $\iota$-ordinary.
\end{enumerate}

\end{theorem}

\begin{proof}
Choose a totally real field $M$ and a prime $\gothl$ of $M$ over $l$, such that $\calO_{M,\gothl}\cong\calO$ and the residue field $\calO_{M,\gothl}/\gothl\calO_{M,\gothl}=k_\gothl$ is isomorphic to $k$. Fix an isomorphism $k\cong k_\gothl$ and view $\overline{\rho}$ as a representation over $k_\gothl$.

Choose a non-CM elliptic curve $E/\QQ$ with good reduction at $l$ and the rational prime $q$ under $\overline{v}_0$. Choose a rational prime $p\ne l$ such that
\begin{itemize}
    \item $p>5$ splits completely in $FM$,
    \item $\SL_2(\FF_p)\subset\overline{r}_{E,p}(G_F)$, $E$ has good reduction at $p$ and $\overline{\rho}$ is unramified at places over $p$.
\end{itemize}
There are a positive density of primes satisfying the first condition by the Chebotarev density theorem, and all but finitely many primes satisfy the second condition, so such $p$ exists. We fix a prime $\gothp$ of $M$ over $p$.

Let $V_{\overline{\rho}}^\vee$ be the $k_\gothl$-vector space scheme over $F$ underlying $\overline{\rho}^\vee$ and fix the standard symplectic pairing on it, which by assumption is preserved by $\overline{\rho}^\vee$ up to multiplier $\overline{\epsilon}_l$. Also we have the $k_\gothp$-vector space scheme $E[p]\cong(E\otimes_\ZZ\calO_M)[\gothp]$ over $F$, equipped with the Weil pairing.

Let $\calD^{-1}=\calD^{-1}_M$ be the inverse different of $M$. We consider the scheme $Y$ over $F$ classifying tuples $(A,j,\alpha_{\overline{\rho}},\alpha_E)$ where
\begin{itemize}
    \item $A$ is an $M$-HBAV with $\calD^{-1}$-polarization $j$,
    \item $\alpha_{\overline{\rho}}:A[\gothl]\to V_{\overline{\rho}}^\vee$ and $\alpha_E:A[\gothp]\to E[p]$ are isomorphisms of vector space schemes compatible with the fixed symplectic pairings on the right hand sides and with the pairings $A[\gothl]\times A[\gothl]\to(\calO_M/\gothl)(1)$ and $A[\gothp]\times A[\gothp]\to(\calO_M/\gothp)(1)$ on the left hand sides.
\end{itemize}
We know by \cite{Tay06} that $Y/F$ is a smooth and geometrically connected variety. Let $X=\Res_{F/F^+}Y$. Then $X/F^+$ is also smooth and geometrically connected.

Applying \cite[Theorem~7.2.4]{ACC+18} (Taylor's potential modularity of elliptic curves) with $\calL=\{l,p\}$ and $L_1^\text{avoid}$ the normal closure of $F^\text{avoid}\overline{F}^{\ker(\overline{\rho}\times\overline{r}_{E,p})}$ over $\QQ$, we get a finite Galois extension $L_2^\text{avoid}/\QQ$ linearly disjoint from $L_1^\text{avoid}$ over $\QQ$ and a finite totally real Galois extension $L^\text{suffices}/\QQ$ unramified above $p$ and $l$ and linearly disjoint from $L_1^\text{avoid}L_2^\text{avoid}$ over $\QQ$, such that for any finite totally real extension $L_2/L^\text{suffices}$ which is linearly disjoint from $L_2^\text{avoid}$ over $\QQ$, there is a regular algebraic cuspidal automorphic representation $\sigma$ of $\GL_2(\AAA_{L_2})$ of weight $0$ such that for any rational prime $p'$ and any isomorphism $\iota_{p'}:\overline{\QQ}_{p'}\cong\CC$, we have $r_{\iota_{p'}}(\sigma)\cong r_{E,p'}^\vee|_{G_{L_2}}$. 

We then apply a theorem of Moret-Bailly in the form stated in \cite[Theorem~3.1]{AN19} to $X$ with
\begin{itemize}
    \item $L=F^+$, $S_1=\{\overline{v}\mid lp\}$, $S_2=\{\overline{v}_0\}$, $L^\text{avoid}=L_1^\text{avoid}L_2^\text{avoid}L^\text{suffices}$,
    \item for $\overline{v}\mid lp$, $\Omega_{\overline{v}}\subset X((F^+_{\overline{v}})^\text{nr})=\prod_{v\mid\overline{v}}Y(F_v^\text{nr})$ is the subset given by abelian varieties $A$ with good reduction at $v$, and furthermore with ordinary reduction at $v$ if $\overline{\rho}^\vee|_{G_{F_v}}$ is good ordinary,
    \item $\Omega_{\overline{v}_0}\subset X(\overline{F^+_{\overline{v}_0}})=\prod_{v_0\mid\overline{v}_0}Y(\overline{F}_{v_0})$ is the subset given by abelian varieties $A$ with good reduction at $v_0$.
\end{itemize}

We need to check that all the assumptions of Moret--Bailly. Real places case is trivial since for each real place $\overline{v}$ of $F^+$, $v$ is the unique complex place of $F$ extending $\overline{v}$, and $X(F^+_{\overline{v}})=Y(F_v)=Y(\CC)$ is clearly non-empty.

For $v\mid p$ in $F$, the subset $\Omega_{\overline{v}}$ is open because having good reduction at $v$ is an open condition and it is Galois invariant because the Galois conjugate of an abelian variety with good reduction at $v$ still has good reduction at $v$, so we only need to show that the subset is nonempty. The two representations $\overline{\rho}$ and $\overline{r}_{E,l}$ over $k_\gothl$ which are unramified at $v$ can be trivialized by passing to some power. More precisely, we choose a positive integer $f$ such that $\overline{\rho}(\Frob_v)^{-f}$ and $\overline{r}_{E,l}(\Frob_v)^f$ are trivial. Now let $A$ be the base change of $E\otimes\calO_M$ to the unramified degree $f$ extension of $F_v$, and let $j$ be induced by the Weil pairing on $E$. By construction of $A$ we know that $j$ is a $\calD^{-1}$-polarization. We have isomorphisms over this extension of $F_v$, $\alpha_{\overline{\rho}}:A[\gothl]\to V_{\overline{\rho}}^\vee$ since both representations are trivial and $\alpha_E:A[\gothp]\to E[p]$ since $p$ splits completely in $M$. It is easy to check compatibility with pairings on both sides. A similar argument shows that $\Omega_{\overline{v}_0}$ is nonempty.

The only thing new here is to check that $\Omega_{\overline{v}}$ is nonempty when $v\mid l$. Applying Lemma \ref{g3.8} to $\overline{\rho}^\vee|_{G_{F_v}}$, we get a divisible $\calO$-module $\calG$ over $\calO_{F_v}$ with an isomorphism of $l$-divisible groups $\lambda:\calG\cong\calG^\vee\otimes\calD_{\calO/\ZZ_l}$ and an isomorphism of finite flat group schemes $i:V_{\overline{\rho}}^\vee\cong\calG[\gothl]_{F_v}$ such that $\lambda$ induces the standard symplectic pairing on $V_{\overline{\rho}}^\vee$. Moreover, $\calG$ is ordinary if $\overline{\rho}^\vee|_{G_{F_v}}$ is good ordinary. Now consider an integral model $\calY/\calO_{F_v}$ of $Y_{F_v}$ classifying tuples $(A,j,\alpha_{\overline{\rho}},\alpha_E)$, where 
\begin{itemize}
    \item $A/S$ ($S$ an $\calO_{F_v}$-scheme) is an $M$-HBAV with $\calD^{-1}$-polarization $j$,
    \item $\alpha_{\overline{\rho}}:A[\gothl]\to\calG[\gothl]$ is an isomorphism of vector space schemes, compatible with the isomorphisms $A[\gothl]\cong A[\gothl]^\vee\otimes\calD_{\calO/\ZZ_l}$ induced by $j$ and $\calG[\gothl]\cong\calG[\gothl]^\vee\otimes\calD_{\calO/\ZZ_l}$ induced by $\lambda$,
    \item $\alpha_E:A[\gothp]\to E[p]$ is an isomorphism of vector space schemes, compatible with the isomorphisms $A[\gothp]\cong A[\gothp]^\vee$ induced by $j$ ($\calD_{\calO_{M,\gothp}/\ZZ_p}$ is trivial since $p$ splits in $M$) and $E[p]\cong E[p]^\vee$ ($E$ has good reduction at $l$ so $E[p]$ extends to a vector space scheme over $\calO_{F_v}$ with the above canonical isomorphism).
\end{itemize}
We need to show that $\calY(\calO_{F_v}^\text{nr})$ is nonempty. By Greenberg's approximation theorem \cite[Corollary~2]{Gre66}, it suffices to show that $\calY(\check{\calO})$ is nonempty, where $\check{\calO}$ is the $l$-adic completion of $\calO_{F_v}^\text{nr}$. Applying Proposition \ref{g3.6} to $\calG_{k_v}$, where $k_v=\calO_{F_v}^\text{nr}/v$ is an algebraically closed field, we get an $M$-HBAV $A_1/k_v$ with $\calD^{-1}$-polarization $j_1$ such that $A_1[\gothl^\infty]\cong\calG_{k_v}$ and $j_1(1)$ coincides with $\lambda_{k_v}:\calG_{k_v}\cong\calG_{k_v}^\vee\otimes\calD_{\calO_{M,\gothl}/\ZZ_l}$ under this isomorphism. Moreover, $A_1$ is ordinary if $\calG_{k_v}$ is ordinary. By Serre--Tate deformation theory, (since naturally we can define $\calG$ over $\check{\calO}$) we can lift $A_1$ to $\widetilde{A}_1$ over $\check{\calO}$ with $\calD^{-1}$-polarization $\widetilde{j}_1$ such that $\widetilde{A}_1[\gothl^\infty]\cong\calG_{\check{\calO}}$ and $\widetilde{j}_1(1)$ coincides with $\lambda_{\check{\calO}}:\calG_{\check{\calO}}\cong \calG_{\check{\calO}}^\vee\otimes\calD_{\calO_{M,\gothl}/\ZZ_l}$ under this isomorphism. By taking $\gothl$-torsion we get $\alpha_{\overline{\rho}}: \widetilde{A}_1[\gothl]\cong\calG_{\check{\calO}}[\gothl]$ and this isomorphism is compatible with the induced pairings on both sides. Finally we let $\alpha_E:\widetilde{A}_1[\gothp]\cong E[p]$ be the isomorphism between two trivial vector space schemes compatible with the pairings on both sides.

Now we checked the assumptions of Moret--Bailly, and hence obtain a finite Galois totally real extension $F_0^+/F^+$ and a point $(A,j,\alpha_{\overline{\rho}},\alpha_E)$ of $X(F_0^+)$ such that
\begin{itemize}
    \item $F_0^+/F^+$ is linearly disjoint from $L_1^\text{avoid}L_2^\text{avoid}L^\text{suffices}$ over $F^+$,
    \item $l$ and $p$ are unramified in $F_0^+$,
    \item $A$ has good reduction above $\overline{v}_0lp$,
    \item if for each $v\mid l$, $\overline{\rho}^\vee|_{G_{F_v}}$ is good ordinary, then $A$ has ordinary reduction at $v$.
\end{itemize}

Set $F_1=F_0^+L^\text{suffices}F$. Then $F_1/F$ is a CM extension unramified above $p$ and $v\mid l$. By \cite[Theorem~3.2]{AN19}, since $F_1^+/L^\text{suffices}$ is a finite totally real extension that is linearly disjoint from $L_2^\text{avoid}$ over $\QQ$, there exists a regular algebraic conjugate self-dual cuspidal automorphic representation $\sigma$ of $\GL_2(\AAA_{F_1})$ of weight 0 (initially of $\GL_2(\AAA_{F_1^+})$, and it extends by solvable base change), such that $\overline{r}_{\iota_p}(\sigma)\cong\overline{r}_{E,p}^\vee|_{G_{F_1}}$ for any $\iota_p:\overline{\QQ}_p\cong\CC$. Moreover $\sigma$ is unramified above $\overline{v}_0lp$ (since $E$ has good reduction above $\overline{v}_0lp$). Since $F_1$ is linearly disjoint from $L_1^\text{avoid}$ over $F$, $\SL_2(\FF_p)\subset\overline{r}_{E,p}(G_{F_1})$ (since $p$ is chosen to satisfy $\SL_2(\FF_p)\subset\overline{r}_{E,p}(G_F)$).

Now we fix a choice of $\iota_p$ and apply the automorphy lifting theorem \cite[Theorem~6.1.1]{ACC+18} to $\rho_0=r_{A,\gothp}^\vee$ (so $\overline{\rho}_0^\vee=A[\gothp]=E[p]$). Then there is a regular algebraic cuspidal automorphic representation $\pi$ of $\GL_2(\AAA_{F_1})$ of weight 0, unramified above $\overline{v}_0lp$, such that $r_{A,\gothp}^\vee\cong r_{\iota_p}(\pi)$. Since the representations $\{r_{A,\gothq}\}_\gothq$ ($\gothq$ ranges over places of $M$) coming from an abelian variety form a compatible system, if we fix $\iota:\overline{\QQ}_l\cong\CC$ we have $r_{A,\gothl}^\vee\cong r_{\iota}(\pi)$, and hence $\overline{r}_{A,\gothl}^\vee\cong\overline{r}_{\iota}(\pi)$. This finishes the proof since we have $\alpha_{\overline{\rho}}:A[\gothl]\cong V_{\overline{\rho}}^\vee$.


Finally we use the Hecke eigenvalue description to show that the automorphic representation $\pi$ is $\iota$-ordinary: say $A$ has ordinary reduction at some $v\mid l$, and the characteristic polynomial of $r_{A,\gothp}^\vee(\Frob_v)$ (which has coefficients in $\calO_M$) has roots $\alpha$, $\beta$, then by \cite[\S2 and \S7]{Del69} we know that one of $\iota^{-1}(\alpha)$, $\iota^{-1}(\beta)$ is an $l$-adic unit. By taking an appropriate $v$-stabilization, we can find a $U_v$-eigenvector in the Iwahori invariants of $\iota^{-1}\pi_v$ with eigenvalue being an $l$-adic unit, and hence by \cite[Definition~4.1.2]{Ger10} we know that $\pi$ is $\iota$-ordinary at $v$.
\end{proof}

\section{Local-Global Compatibility}
We first prove the following property of $\iota$-ordinariness.

\begin{theorem}\label{g2.8}
Let $F$ be a CM field and let $\pi$ be a regular algebraic cuspidal automorphic representation of $\GL_2(\AAA_F)$ of weight $\iota\lambda$, where $\iota:\overline{\QQ}_l\to\CC$ is a fixed isomorphism. Assume that 
\begin{itemize}
    \item $\overline{r}_\iota(\pi)$ is absolutely irreducible and decomposed generic,
    \item $\pi$ is $\iota$-ordinary.
\end{itemize}
Then $r_\iota(\pi)$ is ordinary of weight $\lambda$.
\end{theorem}

\begin{proof}
Fix $w\mid l$. Replacing $F$ with a finite solvable extension in which $l$ splits, we may assume that $F=F^+F_0$ with $F^+\ne\QQ$ totally real and $F_0$ an imaginary quadratic field in which $l$ splits. By \cite[Theorem~5.5.1]{ACC+18}, we know that there exists a Hecke algebra-valued lift $\rho_\gothm$ of $\overline{r}_\iota(\pi)$ such that
\begin{enumerate}
    \item\label{charpoly1} For any $g\in G_{F_w}$, the characteristic polynomial of $\rho_\gothm(g)$ equals $(X-\chi_{\lambda,w,1}(g))(X-\chi_{\lambda,w,2}(g))$,
    \item\label{charpoly2} For any $g_1, g_2\in G_{F_w}$, $(\rho_\gothm(g_1)-\chi_{\lambda,w,1}(g_1))(\rho_\gothm(g_2)-\chi_{\lambda,w,2}(g_2))=0$,
\end{enumerate}
where $\chi_{\lambda,w,i}$ is a Hecke algebra-valued character such that 
$$\chi_{\lambda,w,i}(\Art_{F_w}(\sigma))=\prod_{\tau:F_w\hookrightarrow\overline{\QQ}_l}\tau(\sigma)^{-(\lambda_{\tau,n-i+1}+i-1)}\langle\diag(1,\ldots,\sigma,\ldots,1)\rangle$$
($\sigma$ in the $i$-th place) for $\sigma\in\calO_{F_w}^\times$. Passing to the Hecke eigenvalue on $\pi$, the image of the diamond operator $\langle\diag(1,\ldots,\sigma,\ldots,1)\rangle$ is of finite order, so we can choose an open subgroup of $\calO_{F_w}^\times$ such that it is trivialized. Therefore the $\overline{\QQ}_l$-valued representation $r_\iota(\pi)$ corresponding to $\rho_{\gothm}$ satisfies the properties that are analogous to conditions \ref{charpoly1} and \ref{charpoly2} above, with the characters $\chi_{\lambda,w,i}$ replaced by $\overline{\QQ}_l$-valued characters $\psi_{\lambda,w,i}$, such that
$$\psi_{\lambda,w,i}(\Art_{F_w}(\sigma))=\prod_{\tau:F_w\hookrightarrow\overline{\QQ}_l}\tau(\sigma)^{-(\lambda_{\tau,n-i+1}+i-1)}$$ for all $\sigma$ in some open subgroup of $\calO_{F_w}^\times$.
By Lemma~\ref{BrauerNesbitt} below we know that $r_\iota(\pi)|_{G_{F_w}}$ is conjugate to an upper triangular matrix with (ordered) diagonal entries $\psi_{\lambda,w,1},\psi_{\lambda,w,2}$. This implies the ordinariness of $r_\iota(\pi)$.
\end{proof}

\begin{lemma}\label{BrauerNesbitt}
Let $G$ be a group and let $K$ be an algebraically closed field of characteristic $0$. Fix two characters $\chi_1,\chi_2$. Suppose that $\rho:G\to\GL_2(K)$ is a representation satisfying
\begin{itemize}
    \item For any $g\in G$, the characteristic polynomial of $\rho(g)$ equals $(X-\chi_1(g))(X-\chi_2(g))$,
    \item For any $g_1, g_2\in G$, $(\rho(g_1)-\chi_1(g_1))(\rho(g_2)-\chi_2(g_2))=0$.
\end{itemize}
Then $\rho$ is conjugate to $\left(\begin{array}{cc}
    \chi_1 & \kappa \\
    0 & \chi_2
    \end{array}\right)$ for some $\kappa$.
\end{lemma}

\begin{proof}
The semisimplification of $\rho$ has the same characteristic polynomial as the semisimple representation $\chi_1\oplus\chi_2$. By the Brauer-Nesbitt theorem they are isomorphic, so $\rho$ is conjugate to an upper triangular matrix with (unordered) diagonal entries $\chi_1,\chi_2$. To show that they are put in the right order, suppose that $\rho$ is conjugate to $\left(\begin{array}{cc}
    \chi_2 & \kappa \\
    0 & \chi_1
    \end{array}\right)$.
If $\chi_1=\chi_2$ or $\rho$ is split then we are done. If $\chi_1\ne\chi_2$ and $\rho$ is nonsplit, then $H=\image(\rho)$ is nonabelian (otherwise under some basis every element of $H$ is diagonal). We can find $g_1,g_2\in G$ such that $(\rho(g_1)-\chi_1(g_1))(\rho(g_2)-\chi_2(g_2))$ is nonzero, which contradicts the second condition. Indeed, we choose $g_1$ such that $\chi_1(g_1)=\chi_2(g_1)=1$ and $\kappa(g_1)\ne0$, and choose $g_2$ such that $\chi_1(g_2)\ne\chi_2(g_2)$. Such $g_1$ exists since otherwise the intersection of $H$ and the unipotent subgroup $U$ of the upper triangular Borel $B$ is trivial and hence $H$ embeds into the quotient $B/U$, which is abelian. We can verify that the top right entry of the matrix $(\rho(g_1)-\chi_1(g_1))(\rho(g_2)-\chi_2(g_2))$ is nonzero.
\end{proof}

Now we are ready to prove our main result.

\begin{proof}[Proof of Theorem 1.3]
Fix a prime $p\ne l$ for which $\overline{r}_{\iota}(\pi)$ is decomposed generic. By the main result of \cite{Var14}, it suffices to show that if $v\nmid l$ is a finite place at which $\pi$ is special (i.e., $\pi_v$ is a twist of Steinberg), then $\WD(r_{\iota}(\pi)|_{G_{F_v}})$ has nontrivial monodromy. Let $N$ be the monodromy operator. To show that $N$ is nontrivial it suffices to do so after restricting to a finite extension. In particular, we may go to a solvable base change that is disjoint from $\overline{F}^{\ker(\overline{r}_{\iota}(\pi))}$ in which $p$ splits, and assume that 
\begin{itemize}
    \item $\pi_v$ is an unramified twist of Steinberg,
    \item $\overline{r}_{\iota}(\pi)$ is unramified at $v$ and $v^c$.
\end{itemize}

Now assume that $N=0$. By the main result of \cite{Var14} we have that $r_{\iota}(\pi)|_{G_{F_v}}\cong\chi\oplus\chi\epsilon_l$ for some unramified character $\chi:G_{F_v}\to\overline{\QQ}_l^\times$. We then apply Theorem \ref{g3.9} with $\overline{\rho}=\overline{r}_{\iota}(\pi)$ and $F^\text{avoid}$ equal to the Galois closure of $\overline{F}^{\ker(\overline{r}_{\iota}(\pi))}(\zeta_l)/\QQ$, to get a CM Galois extension $F_1/F$ linearly disjoint from $F^\text{avoid}$ over $F$, such that $\overline{r}_{\iota}(\pi)|_{G_{F_1}}\cong\overline{r}_{\iota}(\pi_1)$ for some weight 0 automorphic representation $\pi_1$ that is unramified above $v$ and $l$ and is $\iota$-ordinary. 

We wish to apply the automorphy lifting theorem \cite[Theorem~6.1.2]{ACC+18}. To do so, we check the following assumptions:
\begin{enumerate}
    \item $r_{\iota}(\pi)|_{G_{F_1}}$ is unramified almost everywhere. We know that $r_{\iota}(\pi)$ is unramified almost everywhere by the main result of \cite{HLTT16} and so is the restriction.
    \item For any $w_1\mid l$ in $F_1$, $r_{\iota}(\pi)|_{G_{F_{w_1}}}$ is ordinary (of some weight $\lambda$ if $\pi$ is of weight $\iota\lambda$). This is by Theorem \ref{g2.8}.
    \item $\overline{r}_{\iota}(\pi)|_{G_{F_1}}$ is absolutely irreducible (encoded in the definition of enormous image) and decomposed generic (by a similar argument in \cite[Theorem~4.1]{AN19}). The image of $\overline{r}_{\iota}(\pi)|_{G_{F_1(\zeta_l)}}$ is enormous (by our choice of $F^\text{avoid}$).
    \item There exists $\sigma\in G_{F_1}-G_{F_1(\zeta_l)}$ such that $r_{\iota}(\pi)(\sigma)$ is a scalar. This is still by our choice of $F^\text{avoid}$.
    \item There exists a regular algebraic cuspidal automorphic representation $\pi_1$ of $\GL_2(\AAA_{F_1})$ and an isomorphism $\iota:\overline{\QQ}_l\cong\CC$ such that $\pi_1$ is $\iota$-ordinary and $\overline{r}_{\iota}(\pi_1)\cong \overline{r}_{\iota}(\pi)|_{G_{F_1}}$. This is by Theorem \ref{g3.9}.
\end{enumerate}

The automorphy lifting theorem gives an $\iota$-ordinary cuspidal automorphic representation $\Pi$ of $\GL_2(\AAA_{F_1})$ such that $r_{\iota}(\Pi)\cong r_{\iota}(\pi)|_{G_{F_1}}$ and $\Pi_{v_1}$ is unramified at all $v_1\mid v$ in $F_1$. Then for any $v_1\mid v$ in $F_1$, $r_{\iota}(\Pi)|_{G_{F_{1,v_1}}}\cong \chi|_{G_{F_{1,v_1}}}\oplus\chi|_{G_{F_{1,v_1}}}\epsilon_l$, and $\Pi_{v_1}$ is an unramified principal series. By local-global compatibility at unramified places \cite{HLTT16} and \cite[Theorem~1]{Var14}, this contradicts the genericity of $\Pi$.
\end{proof}

\bibliographystyle{amsalpha}
\bibliography{bibtex}

\end{document}